\documentclass{amsart}
\usepackage{amsmath, amsthm, amssymb, amsfonts}
\usepackage{bm}
\usepackage{dsfont}
\usepackage[utf8]{inputenc}
\usepackage{rotate}
\usepackage{tikz}
\usepackage{tikz-cd}
\usepackage[arrow,matrix,curve]{xy} 	
\usepackage{xcolor}

\theoremstyle{plain}
\newtheorem{theorem}{Theorem}[section]

\newtheorem{lemma}[theorem]{Lemma}

\theoremstyle{definition}
\newtheorem{definition}[theorem]{Definition}

\newtheorem{remark}[theorem]{Remark}

\newcommand{\norm}[1]{\left\lVert#1\right\rVert}

\newcommand{\vol}{{\rm{vol}}}

\newcommand{\RR}{{\mathbb{R}}}
\newcommand{\ZZ}{{\mathbb{Z}}}
\newcommand{\NN}{{\mathbb{N}}}
\newcommand{\QQ}{{\mathbb{Q}}}

\title{Systems of Rank One, Explicit Rokhlin Towers, and Covering Numbers}
\date{\today}

\author{Christian Wei\ss{}}
\address{{\bf{Ruhr West University of Applied Sciences,}}\\ {{Department of Natural Sciences, Duisburger Str. 100,}}\\{{45479 M\"ulheim an der Ruhr, Germany}}}
\email{christian.weiss@hs-ruhrwest.de}

\begin{document}

\maketitle

\begin{abstract} Rotations $f_\alpha$ of the one-dimensional torus (equipped with the normalized Lebesgue measure) by an irrational angle $\alpha$ are known to be dynamical systems of rank one. This is equivalent to the property that the covering number $F^*(f_\alpha)$ of the dynamical system is one. In other words, there exists a basis $B$ such that for arbitrarily high $h$ an arbitrarily large proportion of the unit torus can be covered by the Rokhlin tower $(f_\alpha^kB)_{k=0}^{h-1}$. Although $B$ can be chosen with diameter smaller than any fixed $\varepsilon > 0$, it is not always possible to take an interval for $B$ but this can only be done when the partial quotients of $\alpha$ are unbounded. In the present paper, we ask what maximum proportion of the torus can be covered when $B$ is the union of $n_B \in \mathbb{N}$ disjoint intervals. This question has been answered in the case $n_B =1$ by Checkhova, and here we address the general situation. If $n_B = 2$ we give a precise formula for the maximum proportion. Furthermore, we show that for fixed $\alpha$ the maximum proportion converges to $1$ when $n_B \to \infty$. Explicit lower bounds can be given if $\alpha$ has constant partial quotients. Our approach is inspired by the construction involved in the proof of the Rokhlin Lemma and furthermore makes use of the Three Gap Theorem. %Rokhlin towers are a standard tool in ergodic theory. However, they are hardly ever made explicit. In this note, we calculate Rokhlin towers for a circle rotation by an irrational angle $\alpha$ and show how they yield bound for covering numbers in the sense of measure-theoretic dynamical systems.
\end{abstract}

\section{Introduction}

If a measure-theoretic dynamical system $(X,T,\mu)$, with $\mu$ being a probability invariant by $T$, is invertible and ergodic, then for arbitrary $\varepsilon > 0$ and $h \in \mathbb{N}$, it is always possible to find a measurable set $B$ such that $B,TB,\ldots,TB^{h-1}$ are disjoint sets and have joint measure greater than $1-\varepsilon$. This result is known as the famous Rokhlin lemma. Its proof is constructive, and hence it allows for finding an explicit set $B$ satisfying the mentioned properties, see Section~\ref{sec:proofs} for more details. The finite sequence of sets $(T^kB)_{k=0}^{h-1}$ is called a \textbf{Rokhlin tower} and $B$ is its \textbf{basis}. In this paper, we exclusively consider a  special class of transformations, namely rotations $f_\alpha: x \mapsto x + \alpha \mod 1$ by an irrational angle $\alpha$ of the unit torus $\mathbb{T}_1 = [0,1)$ with ends of the interval glued, and aim to contribute to a better understanding of Rokhlin towers when an extra condition on the topological structure of the basis $B$ is imposed. The map $f_\alpha$ even yields a uniquely ergodic measure-theoretic dynamical system on the torus where $\mu$ is the Lebesgue-measure. The properties of $f_\alpha$ and also of the structure of its Rokhlin towers has been comprehensively studied, see e.g. \cite{BCF99}, \cite{Che00}, \cite{Fer79}, \cite{Fer97}, \cite{GHL12}, \cite{HK02} to name only a few references.\\[12pt]
Here, we ask how to find \textit{nice} sets $B$ for $f_\alpha$ given $\varepsilon > 0, h \in \mathbb{N}$. In \cite{Che00}, a method based on the Three Gap Theorem (Theorem~\ref{thm:3gap}) was introduced how to calculate the minimal admissible value $\varepsilon$ for $f_\alpha$, if we choose $B$ as an interval, see Theorem~\ref{thm:checkhova}. Amongst others necessary and sufficient conditions on $\alpha$ have been described when $\varepsilon = 0$ can be achieved for arbitrary high $h$. In this paper, we address basis sets consisting of a finite union of intervals and analyze differences and similarities in comparison to the case of only one interval.\\[12pt]
Before we come to the presentation of our results, we mention at first that the dynamical system on $\mathbb{T}_1$ defined by $f_\alpha$ is not only uniquely ergodic (for the Lebesgue measure) but even a so-called system of rank one, see \cite{Fer97}. That means that it does not only satisfy the properties stated in the Rokhlin lemma, but approximates every partition arbitrarily well by Rokhlin towers, which means that the diameter of $B$ can be chosen arbitrarily small. As we have already indicated, this does still not guarantee that $B$ can always be chosen as an interval. For the formal definition of a system of rank one we recall at first that the distance between two partitions $P=\left\{P_1,P_2,\ldots,P_r\right\}, P'=\left\{P_1',P_2',\ldots,P_r'\right\}$ of $X$ is defined by $|P-P'| = \sum \mu(P_i \Delta P_i')$. 
\begin{definition} \label{defi:rank_one} A measure-theoretic dynamical system $(X,T,\mu)$ is of \textbf{rank one} if for any partition $P$ of $X$ and any $\varepsilon > 0$ there exists a subset $B \subset X$, a positive \textbf{height} $h$ and a partition $P'$ of $X$ such that
	\begin{itemize}
		\item $B, TB, \ldots T^{h-1}B$ are disjoint
		\item $|P-P'| < \varepsilon$
		\item $P'$ is refined by the partition made of the sets $B,TB,\ldots,T^{h-1}B$ and $X \setminus \cup_{k=0}^{h-1} T^kB$.
	\end{itemize}
\end{definition}
 In fact, the map $f_\alpha$ has rank one because it is an ergodic transformation of the compact group $\RR / \ZZ$, compare \cite{Vee84}. The notion of covering numbers (see \cite{Kin88}, \cite{Che00}) is closely linked to the definition of systems of rank one, and gives a quantitative measure of how close $(X,T,\mu)$ is to being of rank one. %Still, it cannot be guaranteed from this abstract result that the base set $B$ can be chosen as a topologically \textit{nice} set. %While the definition of systems of rank one is geometric, it does not yield an intuitive \textit{rule} how to find examples of systems of rank one. Yet, there exists an equivalent definition, see Definition~\ref{defi:rank_one_CG}, which is called the constructive-geometric definition of systems of rank one, see \cite{Fer79}, Lemme 16. There are interesting connections between the latter definition and the theory of discrepancy which is however beyond the scope of this paper, see \cite{GHL12}.\\[12pt]
\begin{definition} \label{def:covering_number_general} The \textbf{covering number} $F^*(T)$ of $(X,T,\mu)$ is the supremum of all $z \in \mathbb{R}$ such that for every partition $P = \left\{P_1,\ldots,P_r \right\}$ of $X$ and for every $\varepsilon > 0$ and every $h_0 \in \mathbb{N}$, there exists a subset $B \subset X$ an integer $h \geq h_0$ and a partition $P' = \left\{P_1',\ldots,P_r' \right\}$ of $X$ such that we have
	\begin{itemize}
		\item[(i)] $B, TB, \ldots, T^{h-1}B$ are disjoint,
		\item[(ii)] $\mu(\cup_{k=0}^{h-1}T^kB) \geq z$,
		\item[(iii)] $\sum_{i=1}^r \mu ((P_i \Delta P_i') \cap (\cup_{k=0}^{h-1}T^kB)) < \varepsilon$,
		\item[(iv)] each $P_i'\cap (\cup_{i=0}^{h-1}T^iB)$ is a union of sets $T^jB$, for some $0 \leq j \leq h-1$.		
	\end{itemize}
\end{definition}
The fact that $f_\alpha$ is a system of rank one for any $\alpha \in \mathbb{R} \setminus \mathbb{Q}$ is thus equivalent to $F^*(f_\alpha) = 1$. We call an arc of the torus which is closed on the left and open on the right an \textbf{interval}. If we can approximate any partition, in the sense of Definition~\ref{defi:rank_one}, by a tower whose basis $B$ is an interval, we say that $(X,T,\mu)$ is of \textbf{rank one by intervals}. This idea can be transferred to  Definition~\ref{def:covering_number_general} by restricting covering number to unions of $n_B \in \mathbb{N}$ disjoint intervals and considering $n_B = 1$.
\begin{definition} Let $F_n(T)$ be the supremum of all $z \in \mathbb{R}$ such that for every $h_0 \in \mathbb{N}$, there exists $h \geq h_0$ and a set $B$ consisting of $n_B = n$ disjoint intervals with
	\begin{itemize}
		\item[(i)] $B, TB, \ldots, T^{h-1}B$ are disjoint
		\item[(ii)] $\mu(\cup_{i=0}^{h-1}T^iB) \geq z$.
	\end{itemize}
\end{definition}
Note that $F_1(T) = 1$ means for a transformation $T$ that it is rank one by intervals. For $f_\alpha$, the property $F_1(f_\alpha)=1$, depends on the continued fraction of $\alpha$. Therefore, we briefly fix notation and summarize some of the important properties of continued fractions. For more details, we refer the reader e.g. to \cite{BS96,Nie92}. Let $[a_0;a_1;a_2;\ldots]$ with \textbf{partial quotients} $a_i \in \mathbb{N}_0$ be the continued fraction expansion of $\alpha \in \mathbb{R}$ and denote the corresponding sequence of \textbf{convergents} by $(p_n/q_n)_{n \in \mathbb{N}_0}$. Recall that
\begin{align*} %\label{eq:p}
p_{-2} = 0, p_{-1} = 1, p_n = a_np_{n-1} + p_{n-1}, n \geq 0
\end{align*}
\begin{align*} %\label{eq:q}
q_{-2} = 1, q_{-1} = 0, q_n = a_nq_{n-1} + q_{n-1}, n \geq 0
\end{align*} 
\begin{theorem} \label{thm:f1} Let $\alpha \in \RR \setminus \QQ$. Then $F_1^*(f_\alpha) = 1$ if and only if $\alpha$ has unbounded partial quotients.
\end{theorem}
The result was already mentioned in \cite{Vee84}. A complete proof using Three Gap Theorem is given in \cite{Che00}. For $\alpha \in \mathbb{R} \setminus \mathbb{Q}$ with continued fraction expansion $[a_0;a_1;a_2;\ldots]$ and convergents $p_n,q_n$ a basis of the Rokhlin towers (depending on $\varepsilon$) is given by
	$$B_n = \left[ q_n \left|\frac{p_n}{q_n} - \alpha \right|, \frac{1}{q_n} - q_n \left|\frac{p_n}{q_n} - \alpha \right| \right),$$
compare \cite{Fer97}, Theorem 6. Note that $B_n$ is well-defined for all $n$ if only if $\alpha$ has unbounded partial quotients. In general, the covering number of a base set consisting of $n_B=1$ interval can be precisely calculated by the following theorem.
\begin{theorem}[Checkhova, \cite{Che00}] \label{thm:checkhova} Let $\alpha \in \mathbb{R} \setminus \mathbb{Q}$ have continued fraction expansion $[a_0;a_1,a_2,\ldots]$ with convergents $(p_n,q_n)$ and define $v_n = [0;a_n,a_{n-1},\ldots,a_1]$ and $t_n = [0;a_{n+1},a_{n+2},\ldots]$. Then
$$F_1(f_\alpha) = \lim_{n \to \infty} q_{n+1} |q_n \alpha - p_n| = \limsup_{n \to \infty} \frac{1}{1+t_nv_n}.$$
 Hence, $F_1(f_\alpha) = 1$ if and only if $\alpha$ has unbounded partial quotients and $F_1(f_\alpha) \geq \frac{1}{5 + \sqrt{5}}{10}$ for all $\alpha \in \mathbb{R} \setminus \mathbb{Q}$. The minimal value of $F_1(f_\alpha)$ is attained for any $\alpha \in [0,1) \cap \mathbb{Z}\varphi + \mathbb{Z}$, where $\varphi = \frac{1+\sqrt{5}}{2}$ is the golden mean.
\end{theorem}
If the basis does not only consist of one interval but is a union of several, we may without loss of generality assume that $B$ is (after rotation) of the form
\begin{align} \label{eq:defB}
B = \underbrace{[c_1=0,\beta_1)}_{=:B_1} \cup \underbrace{[c_2, c_2 + \beta_2)}_{=:B_2} \cup \ldots \underbrace{[c_{n_B}, c_{n_B} + \beta_{n_B})}_{=:B_{n_B}}
\end{align}
with $1 \geq c_i \geq c_{i-1} + \beta_{i-1}$ for all $2 \leq i \leq n_B$, and $c_i, \beta_i > 0$. As an analogue of Theorem~\ref{thm:checkhova}, we can precisely calculate the covering number if $n_B = 2$. 
\begin{theorem} \label{thm:F2} Let $\alpha \in \mathbb{R} \setminus \mathbb{Q}$ with continued fraction expansion $[a_0;a_1;a_2;\ldots]$.
\begin{itemize}
    \item[(i)] If $a_i = 1$ eventually, then $F_2(f_\alpha) = \frac{4\varphi+3}{10} = \frac{2\sqrt{5}+5}{10} > F_1(f_\alpha)$.
    \item[(ii)] If $a_i > 1$ eventually, then we have
    $$F_2(f_\alpha) = F_1(f_\alpha).$$
\end{itemize}
\end{theorem}
 In contrast to $F_1(f_\alpha)$ the covering number by two intervals $F_2(f_\alpha)$ does thus not attain its minimum for $\alpha = \varphi = \frac{1+\sqrt{5}}{2}$, because $F_2(f_\varphi) = \frac{2\sqrt{5}+5}{10}$ and for example for $\beta = 1 + \sqrt{2}$ we have $F_2(f_\beta) = F_1(f_\beta) =  1/(4-2\sqrt{2}) < F_2(f_\varphi)$ by Theorem~\ref{thm:checkhova}. The proof of Theorem~\ref{thm:F2} relies on a more deep going application of the Three Gap Theorem (Theorem~\ref{thm:3gap}) than in \cite{Che00}. The reason why this is possible is that the combinatorics behind the orbits of the endpoints of the base intervals can still be controlled if $n_B = 2$. As the number of base interval $n_B$ increases, the situation gets more and more complicated but we can nonetheless give a lower bound when $F_k(f_\alpha) > F_1(f_\alpha)$ holds for a given $\alpha$.
\begin{theorem} \label{thm:weak_lower_bound} We have
$$F_k(f_\alpha) > F_1(f_\alpha)$$
for all $\alpha \in \mathbb{R} \setminus \mathbb{Q}$ with $\limsup_{i \to \infty} a_i \leq k-1$.
\end{theorem}
From Theorem~\ref{thm:F2} and numerical experiments we conjecture that $F_k(f_\alpha) > F_1(f_\alpha)$ actually holds if and only if the partial quotients satisfy $\limsup_{i \to \infty} a_i \leq k-1$. From a quantitative point of view, Theorem~\ref{thm:weak_lower_bound} is nonetheless not very satisfactory because the lower bound therein is constant given $\alpha$. For this reason, we also describe the behaviour of $F_n(f_\alpha)$ as $n \to \infty$. A lower bound in $n$ can even be made explicit if the partial quotients of $\alpha$ are constant. Our proof of this result makes use of the explicit construction of base sets $B$ in the (standard) proof of the Rokhlin Lemma, see Theorem~\ref{thm:rokhlin} and explanations thereafter.
\begin{theorem} \label{thm:bounds:fi} Let $\alpha \in \mathbb{R} \setminus \mathbb{Q}$ have continued fraction expansion $[a_0;a_1;a_2;\ldots]$ and assume that $a_i \leq s$ for all $i \in \mathbb{N}$. Then $F_n(f_\alpha) \to 1$ for $n \to \infty$. Moreover, if $\alpha_i = s$ for all $i \in \mathbb{N}$, then $F_n(f_\alpha)$ can for $n \geq \alpha^{j-1}(\alpha+1)$ be bounded from below by
\begin{align} \label{eq2}
F_n(f_\alpha) \geq F_1(f_\alpha) \frac{1}{\alpha^{j+2}}\left(\lfloor \alpha^{j+1} \rfloor \lfloor \alpha \rfloor + \lfloor \alpha^j \rfloor\right).
\end{align}
\end{theorem}
Note that assuming $a_i \leq s$ in not a real restriction, because in the unbounded case we have $F_1(f_\alpha) = 1$ by Theorem~\ref{thm:f1} and thus $F_n(f_\alpha) = 1$ for all $n \in \mathbb{N}$.
\begin{remark} \label{rem:checkhova} If the continued fraction expansion of $\alpha \in \mathbb{R}$ has constant partial quotients, i.e. $\alpha = [s;s;s;\ldots]$, then $\alpha =  \frac{s+\sqrt{s^2+4}}{2}$, and
$$F_1(f_\alpha) = \frac{1+\alpha^2}{\alpha^2}$$
follows from Theorem~\ref{thm:checkhova}. This implies that the explicit expression for the lower bound of $F_n(f_\alpha)$ in \eqref{eq2} converges to $1$ for $n \to \infty$.
\end{remark}
\section{Proofs of Main Results} \label{sec:proofs}
\paragraph{Three Gap Theorem.} For $\alpha \in \mathbb{R}$, the \textbf{Kronecker sequence} is defined as $(\left\{n\alpha\right\})_{n \in \mathbb{N}}$, where $\left\{ x \right\} := x - \lfloor x \rfloor$ denotes the fractional part of $x \in \mathbb{R}$. The Kronecker sequence may either be interpreted as a sequence on the unit interval $[0,1)$ or, after gluing its endpoints, on the torus. If the latter viewpoint is applied, the norm of a point $x \in \RR$ defined by $\norm{x} := \min(x-\lfloor x \rfloor, 1-(x-\lfloor x \rfloor))$, is used. It measures the distance of $x \in \mathbb{R}$ from the next integer.\\[12pt] 
Most of the results of this paper are based on the so-called Three Gap Theorem which links the gap structure, i.e. the (Euclidean) distance of neighbouring elements of the Kronecker sequence on the torus, to the continued fraction expansion of $\alpha$. Its first proof goes back to S\'os in \cite{Sos58} but it turns out to be useful for us to formulate it here in terms of the Ostrowski expansion, as it is implicitly done in \cite{PSZ16}, see also \cite{Wei20} for a slightly different version. For that purpose let $\alpha \in \mathbb{R} \setminus \mathbb{Q}$ have continued fraction expansion $[a_0;a_1;a_2;\ldots]$ with convergents $p_n/q_n$. Then for $N \in \mathbb{N}$ the Ostrowski expansion of $N$ is uniquely given by
$$N  = \sum_{n=1}^N b_n q_n$$
with $0 \leq b_n \leq a_n$ and $b_{n-1} = 0$ if $b_n = a_n$ .
\begin{theorem}[Three Gap Theorem] \label{thm:3gap} Let $(\left\{n\alpha\right\})_{n \in \mathbb{N}}$ be the Kronecker sequence of $\alpha \in \mathbb{R} \setminus \mathbb{Q}$. Let $N \in \mathbb{N}$ have Ostrowski expansion
$$N  = \sum_{n=1}^N b_n q_n.$$
Then the gaps that can appear have lengths
\begin{align*}
L_1 & = \norm{\left\{ q_n z \right\}},\\
L_2 & = \norm{\left\{ q_{n-1} z \right\}} - (b_n-1)L_1 - \min(b_{n-1},1)L_1,\\
L_3 & = L_1 + L_2,
\end{align*}
and their multiplicities are
\begin{align*}
N_1 & = N - q_n,\\
N_2 & = (b_{n-1}-1)q_{n-1} + \sum_{n=1}^{N-2} b_n q_n,\\
N_3 & = N - N_1 - N_2.
\end{align*}
\end{theorem}
In fact, it is necessary to understand the dynamical behaviour of the Kronecker sequence in order to derive information about the covering numbers of unions of $n_B$ disjoint intervals from it. This can be regarded as an extension of the Three Gap Theorem. Our description here is inspired by the presentation in \cite{Che00} and \cite{Wei20}: If $N = q_i$ is increased to $N'=q_i + q_{i-1} \leq q_{i+1}$, then the former small gaps (for $N$) become large (for $N'$) and every former large gap gets split up into a large and a small gap (both for $N'$). If $a_{i+1} > 2$, every large gap for $N' = q_i + q_{i-1}$ gets split up into a small gap (same length as for $N'$) and a gap of length smaller than the large gap length (for $N'$) but greater than the small gap length (for $N'$) until we reach $N'' = 2q_i + q_{i-1}$. This process is repeated until $N''' = (a_{i+1}-1)q_i + q_{i-1}$. Afterwards, every large gap length (for $N'''$) gets split up into a small gap smaller than for $N'''$ and a now large gap for $\tilde{N}$ (which is equal to the small gap length for $N'$), and this process ends when $\tilde{N} = q_{i+1}$ and then starts all over again. In total, a small gap for $N=q_i$ is split up into $a_{i+1}-1$ small gaps for $\tilde{N}=q_{i+1}$ and $1$ large gap for $\tilde{N}=q_{i+1}$ when passing from $q_i$ to $q_{i+1}$. Similarly, a large gap for $N=q_{i}$ is split up into $a_{i+1}$ small gaps for $\tilde{N}=q_{i+1}$ and $1$ large gap for $\tilde{N}=q_{i+1}$.\\[12pt]
\paragraph{Covering Numbers of Rotations.} These observations imply four fundamental but important facts for disjoint orbits of sets of the form described in \eqref{eq:defB} with $n_B = 2$. 
\begin{lemma} \label{lem:basic_properties_2int3} Let $\alpha \in \mathbb{R} \setminus \mathbb{Q}$ with continued fraction expansion $[a_0;a_1;a_2;\ldots]$ and convergents $p_n,q_n$. Let $B$ be a set of the form \eqref{eq:defB} and let $q_n < N \in \mathbb{N} \leq q_{n+1}$ be arbitrary. If $B, f_\alpha B, \ldots f_\alpha^{N-1}B$ are disjoint then %Furthermore, let $g_S$ denote the smallest gap length occurring in $\left\{0 \alpha \right\}, \left\{1 \alpha \right\}, \ldots \left\{(N-1) \alpha \right\}$. 
\begin{itemize}
	\item[(i)] $\beta_i \leq \norm{q_n \alpha}$ holds for all $\beta_i$.
	\item[(ii)] If $n_B = 2$ and $a_{n+1} > 1$, then $B$ can be chosen such that $\beta_1 + \beta_2 = 2\norm{q_n \alpha}$ if and only if $q_n < N \leq \frac{1}{2}q_{n+1}$.
	\item[(iii)] If $n_B = 2$ and $a_{n+1} > 1$ and $\norm{q_n \alpha} < \beta_1 + \beta_2 < 2\norm{q_n \alpha}$, then $\beta_1 + \beta_2 \leq \norm{q_{n-1}\alpha} - (a_{n+1}-1)\norm{q_n\alpha}$ and $q_{n+1} < 2N \leq q_{n+1}+q_n$. If $q_{n+1} < 2N \leq q_{n+1} + q_n$, then there exists a $B$ with $\beta_1+\beta_2 = \norm{q_{n-1}\alpha} - (a_{n+1}-1)\norm{q_n\alpha}$.
	\item[(iv)] If $n_B=2$ and $a_{n+1} = 1$ and $\norm{q_n \alpha} < \beta_1 + \beta_2$, then $\beta_1 + \beta_2 \leq \norm{q_{n+1}\alpha} + \norm{q_{n+2}\alpha} = \norm{q_n\alpha}$ and $q_{n} < N \leq q_{n} + q_{n-1}/2$. If $q_{n} < N \leq q_{n} + q_{n-1}/2$, then there exists a $B$ with $\beta_1+\beta_2 = \norm{q_n\alpha}$.
\end{itemize}
\end{lemma}
Observe that if $a_{n+1} = 1$, then the situation described in part (ii) cannot occur, because the long gap length is less than twice the small one in this case.
\begin{proof} Assertion (i) is immediately clear by the Three Gap Theorem since the orbits of the intervals $B_i$ have to be disjoint.\\[12pt] 
For the second assertion, the following line of argument leads to the desired result: let us assume first that $a_{n+1} > 2$ is even. If $N = q_{n} + q_{n-1}$ then there are $q_{n-1}$ small gaps of size $\norm{q_n \alpha}$ and $q_n$ large gaps of size $\norm{q_{n-1}\alpha}$. At most $a_{n+1}$ small gaps fit into a large gap. In order to have $\beta_1 = \beta_2 = \norm{q_n \alpha}$, there needs to be an extra small space on the right of each block of small gaps. However, for $N = a_{n+1}/2 \cdot q_n$ the gap structure of the sequence is as follows: it consists solely of blocks with $(a_{n+1}/2-1)$ small gaps followed either by a medium or a large sized gaps. In each of the following $j$ steps one of the large gaps gets, as usual, split up into a medium and a small size gap, i.e. there are $q_{n-1}-j$ large gaps while we get $j$ additional small gaps. Therefore, we must have $2j \leq q_{n-1}$. This condition implies $2N \leq q_{n+1}$ and is necessary for $\beta_1+\beta_2 = 2 \norm{q_n \alpha}$. If $a_{n+1} > 1$ is odd then for $N = \lfloor a_{n+1}/2 \rfloor \cdot q_n$ the gap structure of the sequence is as follows: it consists solely of blocks with $(\lfloor a_{n+1}/2 \rfloor-1)$ small gaps followed either by a medium or a large sized gaps and a similar argument as in the even case can be applied. Now let $N$ satisfy the assumptions of (ii) and choose $\beta_1 = \beta_2 = \norm{q_n \alpha}$ and $c_1 = \norm{N\alpha}$. The orbit of the left endpoints of $B$ thus equals the Kronecker sequence $\alpha,\ldots,N\alpha,(N+1)\alpha,\ldots,2N\alpha$. As $2N \leq q_{n+1}$, the joint sequence has minimal gap length $\norm{q_n \alpha}$ and thus the orbit of $B$ is disjoint. Finally, in the case $a_{n+1} = 2$ and $N = q_n$, there are $q_{n-1}$ large gaps and $q_n-q_{n-1}$ small gaps and the gap sizes satisfy $\norm{q_{n-1}\alpha} < \norm{q_{n-2}\alpha} < 2\norm{q_{n-1}\alpha}$. Therefore $N \leq q_n + q_{n-1}/2 = \frac{1}{2}q_{n+1}$ must hold if $\beta_1 = \beta_2 = \norm{q_n \alpha}$. Again, the upper bound on $\beta_1 + \beta_2$ can be realized by a rotation with angle, i.e. $c_1 = \norm{N\alpha}$\\[12pt]
If $N \leq \frac{1}{2}q_{n+1}$, then we are in the situation of (ii). So we may restrict to the case $N > \frac{1}{2}q_{n+1}$ when we want to prove (iii). The maximal possible value for $\beta_1 + \beta_2$ is indeed $\norm{q_{n-1}\alpha} - (a_{n+1}-1)\norm{q_n\alpha}$ and this is the only possibility because $\norm{q_{n-1}\alpha} - a_{n+1} \norm{q_n\alpha} < \norm{q_n\alpha}$. Note that in the $\lfloor q_{n}/2 \rfloor$ steps between the upper and the lower bound for $N$ every medium and large gap gets diminished by $\norm{q_n \alpha}$. Thus it follows that $\beta_1 + \beta_2 = \norm{q_{n-1}\alpha} - (a_{n+1}-1)\norm{q_n\alpha}$ implies $2N \leq q_{n+1} + q_n$. A rotation with angle, i.e. $c_1=\norm{N\alpha}$, yields the claim as for (ii).\\[12pt]
In order to prove (iv) we only note that $\beta_1+\beta_2$ can be realized by a rotation with angle $c_1 = \norm{N\alpha}$ because $2N \leq 2q_n + q_{n-1} = q_{n+1} + q_n$ and the Three Gap Theorem.  Counting the number of large gaps yields the upper bound on $N$.
%If $N > \lfloor \tfrac{a_{n+1}}{2} \rfloor q_n + q_{n-1}$, then there are blocks of $\tfrac{a_{n+1}}{2}+1$ small intervals. These blocks do not anymore fit into one of the remaining large gaps (because they are too small). Hence we have $\beta_1 + \beta_2 = g_S$ if the orbit consists of disjoint sets. Similarly as in (iii), the maximal remaining space is $g_L - \tfrac{a_{n+1}}{2} g_S$ under the conditions of (iv).  Assertion (v) follows trivially as (ii)--(iv).  
%For the second assertion let $C_1,\ldots,C_N$ be the gaps of the Kronecker sequence. If $N > 3$, at least one of the medium or small gaps must contain at least one point of any translated Kronecker sequence. Hence the claim follows. If $N=3$ and if there are two large and one medium gaps, then it can be checked by a simple calculation that (ii) is still true.
\end{proof}
Theorem~\ref{thm:F2} can be deduced by using Lemma~\ref{lem:basic_properties_2int3}.
\begin{proof}[Proof of Theorem~\ref{thm:F2}] Let us at first take a look at the general setting for any set 
$$B = [0,\beta_1) \cup [c_2, c_2 + \beta_2)$$
consisting of two intervals. By Lemma~\ref{lem:basic_properties_2int3}~(iii) and (iv), it follows for $2N > q_n + q_{n-1}$ that 
$$\vol (B \cup f_\alpha B \cup \ldots \cup f_\alpha^{N-1}B) \leq q_{n+1} \norm{q_n \alpha} \leq F_1(f_\alpha).$$ 
If $2N < q_{n+1}$,  Lemma~\ref{lem:basic_properties_2int3}~(ii) yields
\begin{align*}
    \vol (B \cup f_\alpha B \cup \ldots \cup f_\alpha^{N-1}B) = 2 \norm{q_n \alpha} N  < q_{n+1} \norm{q_n \alpha} \leq F_1(f_\alpha).
\end{align*}
The cases which remain to be checked are the maximal values of $N$ in Lemma~\ref{lem:basic_properties_2int3}~(iii) and (iv).\\[12pt]
(i) If $a_n=1$ eventually and $N=q_{n+1} + \lfloor \tfrac{1}{2}q_n \rfloor$, then
\begin{align*}
    \vol (B & \cup f_\alpha B \cup \ldots \cup f_\alpha^{N-1}B) = \norm{q_n \alpha} N \\
    & = (q_{n+1}+\lfloor\tfrac{1}{2}q_n\rfloor) \norm{q_n \alpha} 
\end{align*}
follows by Lemma~\ref{lem:basic_properties_2int3}~(iv). The right hand side of the equality converges to $$\frac{1}{1+\frac{1}{\theta}\frac{1}{\theta}} + \frac{1}{2} \frac{1}{\theta} \frac{1}{1+\frac{1}{\theta}\frac{1}{\theta}} = \frac{4\theta+3}{10} = \frac{2\sqrt{5}+5}{10}.$$
(ii) For $a_{n+1} \geq 3$ and $N=\lfloor \tfrac{1}{2} (q_n+q_{n+1}) \rfloor$, we get
%\begin{align*}
%    \vol (B & \cup f_\alpha B \cup \ldots \cup f_\alpha^{N-1}B) = (\norm{q_n \alpha}+\norm{q_{n+1}\alpha}) N \\
%    & = (q_{n+1}+\lfloor\tfrac{1}{2}q_n\rfloor) \norm{q_n \alpha}. 
%\end{align*}
\begin{align*}
    \vol (B & \cup f_\alpha B \cup \ldots \cup f_\alpha^{N-1}B) = (\norm{q_n \alpha}+\norm{q_{n+1}\alpha}) N \\
    & \leq \frac{1}{2}(q_{n+1}+q_n) (\norm{q_n \alpha} + \norm{q_{n+1}\alpha})\\ & \leq \frac{1}{2}F_1(f_\alpha) + \frac{1}{6}F_1(f_\alpha) + \frac{1}{6}(f_\alpha) + \frac{1}{9}F_1(f_\alpha) < F_1(f_\alpha).
\end{align*}
For $a_{n+1} =2$ and $N=\lfloor \tfrac{1}{2} (q_n+q_{n+1}) \rfloor$, Lemma~\ref{lem:basic_properties_2int3}~(iii) implies
\begin{align*}
    \vol (B & \cup f_\alpha B \cup \ldots \cup f_\alpha^{N-1}B) = (\norm{q_n \alpha}+\norm{q_{n+1}\alpha}) N \\
    & = \frac{1}{2} (q_{n+1} + q_n) (\norm{q_n \alpha}+\norm{q_{n+1}\alpha}). \\
    & = \frac{1}{2} q_{n+1} \norm{q_n\alpha} \left(1 + \frac{q_n}{q_{n+1}} + \frac{\norm{q_{n+1}\alpha}}{\norm{q_n\alpha}} + \frac{q_n}{q_{n+1}} \frac{\norm{q_{n+1}\alpha}}{\norm{q_n\alpha}}\right)
\end{align*}
If $a_{n+1} = 2$ eventually, then $q_{n}/q_{n+1} \to \frac{1}{2+\sqrt{2}} = \sqrt{2}-1$ and $\norm{q_{n+1}\alpha}/\norm{q_n\alpha} \to \frac{1}{2+\sqrt{2}}$. If $a_{n+1} = 2$ and $a_n > 2$, then $q_n/q_{n+1} < \frac{1}{2+\sqrt{2}}$. As $(1 +\frac{1}{2+\sqrt{2}})^2 = 2$, we have in any case $\lim_{N \to \infty} \vol(B \cup f_\alpha B \cup \ldots \cup f_\alpha^{N-1}B) \leq F_1(f_\alpha)$. In conclusion, we always get $F_2(f_\alpha) \leq F_1(f_\alpha)$.
\end{proof}
From the observation in the proof of Lemma~\ref{lem:basic_properties_2int3} that the choice $c_1 = N\alpha$ suffices to attain the maximum in the case of two intervals, it is possible to also derive a lower bound on $F_k$ in the case $a_i < k$ for all but finitely many $i\in \mathbb{N}$. This is the content of Theorem~\ref{thm:weak_lower_bound}.
\begin{proof}[Proof of Theorem~\ref{thm:weak_lower_bound}] Let $k^* = \limsup_{i \to \infty} a_i$. By considering the corresponding subsequence we may without loss of generality assume $a_i = k^*$ for all $i \in \mathbb{N}$. Partitioning the complete sequence Kronecker sequence $z = \left\{ j\alpha \right\}_{j=1}^{q_{n+1}}$ into $k^*$ blocks of elements of size $\lfloor q_{n+1}/k^* \rfloor$ yields
$$F_{k^*+1}(f_\alpha) \geq (k^*-1) \lfloor \frac{q_{n+1}}{k^*} \rfloor \norm{q_n \alpha} + \lfloor \frac{q_{n+1}}{k^*} \rfloor (\norm{q_{n}\alpha} + \norm{q_{n+1}\alpha}),$$
where the length $\norm{q_{n}\alpha} + \norm{q_{n+1}\alpha}$ stems from the last block.
\end{proof}
\paragraph{Explicit Rokhlin Towers.} In order to improve the lower bound for the area covered by $B,f_\alpha B,\ldots,f_\alpha^{h-1}B$ for circle rotations by $\alpha \in \mathbb{R} \setminus \mathbb{Q}$ we will make use of the construction used in the standard proof of the Rokhlin lemma, as it is presented e.g. in \cite{EW11}~Lemma 2.45. %In this paper, we are mainly interested in these sets $B$ and will therefore state the Rokhlin lemma first.
\begin{theorem}[Rokhlin Lemma] \label{thm:rokhlin} Let $(X,T,\mu)$ be an invertible ergodic measure-theoretic dynamical system with non-atomic $\mu$. Then for any $h \geq 1$ and $\varepsilon > 0$, there is a measurable set $B$ such that $B,TB,\ldots,T^{h-1}B$ are disjoint and
$$\mu(B\cup TB \cup \ldots \cup T^{h-1}B) > 1 - \varepsilon.$$
\end{theorem}
We therefore shortly recall the construction in the proof of Theorem~\ref{thm:rokhlin} at first. Choose a measurable set $A$ with $0<\mu(A)<\varepsilon/h$ and define the sets
\begin{align*}
A_1 &= A \cap T^{-1}A,\\
A_2 & = A \cap T^{-2}A \setminus A_1,\\
& \vdots\\
A_n & = A \cap T^{-n}A \setminus \setminus \bigcup_{i<n} A_i.
\end{align*}
The sets $A_i$ are disjoint and so is the union
$$A_k \cup TA_k \cup \ldots \cup T^{k-1}A_k$$
for all $k \geq 1$. Then the base set $B$ defined by
$$B = \bigcup_{k \geq n} \bigcup_{j=0}^{\lfloor k/n \rfloor -1} T^{jn}A_k$$
satisfies the properties mentioned in Lemma~\ref{thm:rokhlin}.\\[12pt]
Now let $\alpha \in \mathbb{R} \setminus \mathbb{Q}$ and let $p_n,q_n$ be the sequence of the convergents. According to classical continued fraction theory $\norm{q_n \alpha} \to 0$ for $n \to \infty$ holds. For $\varepsilon > 0$ and $n \in \mathbb{N}$ arbitrary, set $\varepsilon' := \varepsilon / n$. In order to apply the following algorithm we furthermore need to impose the condition that $\varepsilon' < 1$ (although this might seem to be a trivial remark, it will indeed have some importance in the remainder of the proof).
Assume that $\norm{q_k\alpha} < \varepsilon' \leq \norm{q_{k-1}\alpha}$. Given $a_{k+1}>1$, the value $\varepsilon'$ lies in one of the following intervals
\begin{align*} & \underbrace{(\norm{(q_k+q_{k-1})\alpha},\norm{q_{k-1}\alpha}]}_{:=I_0}, \quad \underbrace{(\norm{(2q_k+q_{k-1})\alpha},\norm{(q_k+q_{k-1})\alpha}]}_{:=I_1}, \ldots, \\
& \underbrace{(\norm{q_{k}\alpha}),\norm{((a_{k+1}-1)q_k+q_{k-1})\alpha}]}_{:=I_{a_{k+1}}}.
\end{align*}
If $\left\{q_k\alpha\right\} > 0.5$, then $\left\{(iq_k + q_{k-1})\alpha \right\} < 0.5$ for $1 \leq i \leq a_{k+1}$. If $\varepsilon' \in I_{j-1}$, let $l_0(k) := q_k, l_1(k) := jq_k + q_{k-1}$ and $l_2(k) := (j+1)q_k + q_{k-1}$ and $\alpha_{k,0} := \left\{l_0(k)\alpha\right\}, \alpha_{k,1} = \left\{l_1(k)\alpha\right\}$ and $\alpha_{k,2} = \left\{ l_2(k)\alpha \right\}$. Define the sets $A_i$ as in the proof of the Rokhlin Lemma. These intervals are explicitly given by
\begin{align*}
    A_1 & = A \cap f_\alpha^{-1}A = \emptyset\\
     & \vdots\\
    A_{l_0(k)-1} & = A \cap f_\alpha^{-(l_0(k)-1)}A = \emptyset\\
    A_{l_0(k)} & = A \cap f_\alpha^{-l_0(k)}A = [\alpha_{k,0},\varepsilon') \\
    A_{l_0(k)+1} & =  A \cap f_\alpha^{-(l_0(k)+1)}A = \emptyset\\
     & \vdots\\
    A_{l_1(k)-1} & = A \cap f_\alpha^{-(l_1(k)-1)}A = \emptyset\\
    A_{l_1(k)} & = A \cap f_\alpha^{-l_1(k)}A = [0,\varepsilon'-\alpha_{k,1}) \\
    A_{l_1(k)+1} & =  A \cap f_\alpha^{-(l_1(k)+1)}A = \emptyset\\
     & \vdots\\
    A_{l_2(k)-1} & = A \cap f_\alpha^{-(l_2(k)-1)}A = \emptyset\\
    A_{l_2(k)} & = A \cap f_\alpha^{-l_2(k)}A = [\varepsilon'-\alpha_{k,1},\alpha_{k,0})
\end{align*}
and $A_i = \emptyset$ for all $i > l_2(k)$. Note that $A_{l_1(k)} = \emptyset$ might happen (if the left and the right endpoint of the interval are equal). More precisely, we observe that only two intervals occur excatly if $\varepsilon'$ is a right endpoint of one of the $I_j$. For the cases $\left\{q_k\alpha\right\} < 0.5$ and $a_{i+1} = 1$, the calculation of the $A_i$ can be treated similarly. Thus, we may without loss of generality restrict our analysis to the case $\left\{q_k\alpha\right\} > 0.5$, which was described above explicitly.\\[12pt]
If $n=q_{k+1}$ and $\varepsilon' = \norm{q_k\alpha}$, then it can be easily checked that $l_0(k) = q_{k+1}, l_1(k) = q_{k+1} + q_k$ and these are the only non-empty sets $A_i$. Since $\lfloor l_1(k) / n \rfloor = 0$, it follows that $B$ consists of one interval only. Moreover, the Rokhlin tower built by the base $B$ satisfies $\limsup_{n \to \infty} \vol(B \cup f_\alpha B \cup \ldots \cup f_\alpha^{n-1}B) = F_1(f_\alpha)$. Thus, we see that for $\varepsilon = q_{k+1} \norm{q_k\alpha} < 1$, the construction from the Rokhlin lemma yields a basis consisting of one interval only which covers the maximal possible area.\\[12pt]
Writing $\varepsilon := n \varepsilon'$, our idea is now to keep $n$ constant while alternating $\varepsilon'$. Due to our preparatory work the actual proof of Theorem~\ref{thm:bounds:fi} can be kept relatively short now. %If $n = q_{k+1}$, we see that $\varepsilon < 1$ implies $\varepsilon' \in I_j(k)$. 
\begin{proof}[Proof of Theorem~\ref{thm:bounds:fi}] Let $\varepsilon' = \norm{q_{k+j}\alpha}$ for some $j \in \mathbb{N}_0$. The intervals $A_{l_0(k+j)}$ and $A_{l_1(k+j)}$ are adjacent by construction. Therefore the basis $B$ consists of $\lfloor l_1(k+j)/n \rfloor = \lfloor (q_{k+j+1} + q_{k+j})/n \rfloor$ distinct intervals. Hence $\lfloor l_1(k+j)/n \rfloor \leq \alpha_s^{j-1}(\alpha_s+1)$ for $j \gg 1$ big enough. The area covered by $B,f_\alpha B,\ldots,f_\alpha^{n-1}B$ is equal to
%\begin{align*}
%q_{k+1} & \cdot \left( \lfloor q_{k+j+1}/q_{k+1} \rfloor \cdot (\norm{q_{k+j}\alpha}-\norm{q_{k+j+1}\alpha}) + \lfloor q_{k+j}/q_{k+1} \rfloor \cdot \norm{q_{k+j+1}\alpha} \right) \\
%& \geq \tilde{r}^{j} \cdot \ldots + \tilde{r}^{j-1} \cdot \ldots
%&> q_{k+j+1} \cdot \left( \frac{1}{q_{k+j+1}+q_{k+j}} - \frac{1}{q_{k+j+2}}\right) + q_{k+j} \cdot \frac{1}{q_{k+j+1}+q_{k+j+2}}\\
%&\geq \frac{1}{1+s} \frac{1}{\tilde{s}+\tilde{s^2}}
%\end{align*}
\begin{align} \label{eq1}
    q_{k+1} \lfloor q_{k+j+1} / q_{k+1} \rfloor \norm{q_{k+j} \alpha} + q_{k+1} \lfloor q_{k+j} / q_{k+1} \rfloor \norm{q_{k+j+1}\alpha}.
\end{align}
From the Three Gap Theorem~\ref{thm:3gap}, it follows that $q_{k+j+1}\norm{q_{k+j}\alpha} + q_{k+j}\norm{q_{k+j+1}\alpha}=1$. Therefore, expression \eqref{eq1} converges to $1$ as $j \to \infty$ which implies $F_n(\alpha) \to 1$ as $n \to \infty$. If $a_i = s$ for all $i \in \mathbb{N}$, then \eqref{eq1} can be
bounded from below by
\begin{align*}
    q_{k+1} \cdot & \left( \lfloor \alpha^{j+1} \rfloor \norm{q_{k+j}\alpha} + \lfloor \alpha^j \rfloor \norm{q_{k+j+1}\alpha} \right) \geq q_{k+1} \cdot \norm{q_{k+j+1}\alpha}\left(\lfloor \alpha^{j+1} \rfloor \lfloor \alpha \rfloor + \lfloor \alpha^j \rfloor\right)\\ & \geq q_{k+1} \cdot \norm{q_k\alpha} \frac{1}{\alpha^{j+2}} \left(\lfloor \alpha^{j+1} \rfloor \lfloor \alpha \rfloor + \lfloor \alpha^j \rfloor\right). 
\end{align*}
Applying $\limsup$ with respect to $k$ implies the desired result.
\end{proof}
\section{Constructive Geometric Definition of Systems of Rank One.} At the end of this paper, we would like to comment on an alternative definition of systems of rank one, namely the constructive geometric one, and thereby round off the presentation of the topic. An advantage of Definition~\ref{defi:rank_one_CG} in comparison to Definition~\ref{defi:rank_one} is that the bullet points in Definition~\ref{defi:rank_one_CG} give an explicit possibility how to define a system of rank one and thereby yield an infinite class of examples. A proof for the equivalence of the two definitions can be found e.g. in \cite{Fer79}, Lemme~16.
\begin{definition} \label{defi:rank_one_CG} A dynamical system $(X,T,\mu)$ is of rank one if there exists a sequence of positive integers $q_n, n \in \NN$ and $a_{n,i}, n \in \NN, 1 \leq i \leq q_n -1$ such that if $h_n$ is defined by
	$$h_0 = 1, \qquad h_{n+1} = q_nh_n + \sum_{i=1}^{n-1}a_{n,i}$$
then
	$$\sum_{n=0}^\infty \frac{h_{n+1}-q_nh_n}{h_{n+1}} < \infty,$$
and subsets $B_n \subset X, n \in \NN, B_{n,i}, n \in \NN, 1 \leq i \leq q_n$ and $C_{n,i,j}, n \in \NN, 1 \leq i \leq q_{n}-1, 1 \leq j \leq a_{n,i}$ such that for all $n$
\begin{itemize}
	\item the $B_{n,i}, 1 \leq i \leq q_n$ form a partition of $F_n$,
	\item the $T^kB_n, 1 \leq k \leq h_{n}-1$ are disjoint,
	\item $T^{h_n}B_{n,i} = C_{n,i,1}$ if $a_{n,i} \neq 0$ and $i < q_n$,
	\item $T^{h_n}B_{n,i} = B_{n,i+1}$ if $a_{n,i} = 0$ and $i < q_n$,
	\item $TC_{n,i,j} = C_{n,i,j+1}$ if $j < a_{n,i}$,
	\item $TC_{n,i,a_{n,i}} = B_{n,i+1}$ if $j < a_{n,i}$,
	\item $B_{n+1} = B_{n,1}$	
\end{itemize}
and the partitions $\left\{ B_n, TB_n, \ldots, T^{h_n-1}B_n, X \setminus \cup_{k=0}^{h_n-1}T^k B_n \right\}$ converge to the Lebesgue $\sigma$-Algebra of $X$.
\end{definition}
We call a rank one map \textbf{(CG) rank one by intervals}, if the sets $B_n$ can be chosen as intervals. A detailed description of the action of $T$ for the case that $X$ is the unit torus can be found e.g. in \cite{GHL12}.\\[12pt] 
If, given a system of rank one in the sense of Definition~\ref{defi:rank_one}, we were able to find the sets and calculate the coefficients in Definition~\ref{defi:rank_one_CG}, one of the main outcomes would be that the star-discrepancy (see \cite{Nie92}) of $T$-orbits could be easily calculated, see \cite{DP99}, Th\'eor\`{e}me 2.6, or \cite{GHL12}, Theorem 86. On the downside of Definition~\ref{defi:rank_one_CG}, this task seems to be very hard also for dynamical systems for which it is comparably easy to prove that they are of rank one in the sense of Definition~\ref{defi:rank_one} (despite that the proof of the equivalence of the two definitions is even partially though not completely constructive, compare \cite{Fer79}, Lemme~16.). At the end of this article, we give a theoretical justification what makes this task particularly hard for rotations. The reason is that the sets $B_n$ can never be chosen as intervals.
\begin{theorem} For $\alpha \in \RR \setminus \QQ$ arbitrary, the rotation map $f_\alpha: x \mapsto x + \alpha$ is not (CG) rank one by intervals. 
\end{theorem}
\begin{proof} Assume that the claim is false and let $B_1$ be an interval. By rotating (if necessary) we may without loss of generality assume that $B_1 = [0,x)$ for some $x > 0$. By definition $f_\alpha^i(B_2)$ would then be of the form $[0,x/k)$ with $k \in \NN$ for some $i \in \NN$. Since the finite sequence $(f_\alpha^l(B_2))$ needs to build a partition of $B_1$ we must have
	$$\frac{\left\{n \alpha \right\}}{\left\{m \alpha \right\}} = k$$	
for some $m,n \in \mathbb{N}$ which can only hold for $n = mk$. This contradicts the fact that $x$ is in the forward-orbit of $x/k$. Thus, the rotation cannot be (CG) rank one by intervals.
\end{proof}

\paragraph{Acknowledgment.} Research on this paper started in 2020 during the trimester program \textit{Dynamics: Topology and Numbers} at the Hausdorff Research Institute for Mathematics in Bonn whom I would like to thank for hospitality. Moreover I would like to thank S\'ebastien Ferenczi for useful discussions on the topic of this paper.
\bibliographystyle{alpha}
\bibdata{references}
\bibliography{references}
\end{document}